\documentclass[12pt]{amsart}
\usepackage{amsmath,amsfonts,amstext,amsthm,epsfig}

\newtheorem{Theorem}{Theorem}[section]
\newtheorem{Lemma}[Theorem]{Lemma}
\newtheorem{Prop}[Theorem]{Proposition}
\newtheorem{Rem}[Theorem]{Remark}
\newtheorem{Def}[Theorem]{Definition}
\newtheorem{Hyp}[Theorem]{Hypothesis}
\newtheorem{Notat}{Notation}
\makeatletter
    \addtocounter{section}{0}
    
     \@addtoreset{equation}{section}
\makeatother

\def\N{\mathbb N}
\def\R{\mathbb R}

\begin{document}

\title[Dimension of invariant sets]
{Dimension of attractors and invariant sets of damped wave equations in unbounded domains}%
\author{Martino Prizzi}

\address{Martino Prizzi, Universit\`a di Trieste, Dipartimento di
Matematica e Informatica, Via Valerio 12/1, 34127 Trieste, Italy}%
\email{mprizzi@units.it}%
\subjclass{35L70,
 35B41 }%
\keywords{damped wave equation, invariant set, attractor, dimension}%

\date{\today}%
\begin{abstract} Under fairly general assumptions, we prove that every compact invariant set $\mathcal I$ of the semiflow generated
by the semilinear damped wave equation
\begin{equation*}
\begin{aligned}
u_{tt}+\alpha u_t+\beta(x)u-\Delta u&=f(x,u),&&(t,x)\in[0,+\infty[\times\Omega,\\
u&=0,&&(t,x)\in[0,+\infty[\times\partial\Omega
\end{aligned}\end{equation*}
in $H^1_0(\Omega)\times L^2(\Omega)$ has finite Hausdorff and
fractal dimension. Here $\Omega$ is a regular, possibly unbounded,
domain in $\R^3$ and $f(x,u)$ is a nonlinearity of critical
growth. The nonlinearity $f(x,u)$ needs not to satisfy any
dissipativeness assumption and the invariant subset $\mathcal I$
needs not to be an attractor. If $f(x,u)$ is dissipative and
$\mathcal I$ is the global attractor, we give an explicit bound on
the Hausdorff and fractal dimension of $\mathcal I$ in terms of
the structure parameters of the equation.
\end{abstract}
\maketitle

\section{Introduction}
In this paper we consider the damped wave equation
\begin{equation}\label{equation1}
\begin{aligned}
u_{tt}+\alpha u_t+\beta(x)u-\Delta
u&=f(x,u),&&(t,x)\in[0,+\infty[\times\Omega,\\
u&=0,&&(t,x)\in[0,+\infty[\times\partial\Omega
\end{aligned}\end{equation}

Here $\Omega$ is a regular (possibly unbounded) open set in
$\R^3$, $\beta(x)$ is a potential such that the operator
$-\Delta+\beta(x)$ is positive, and $f(x,u)$ is a nonlinearity of
critical growth (i.e. of polynomial growth less than or equal to three). The assumptions on $\beta(x)$ and $f(x,u)$ will be made more
precise in Section 2 below. Under such assumptions, equation
(\ref{equation1}) generates a local semiflow $\Pi$ in the space
$H^1_0(\Omega)\times L^2(\Omega)$. Suppose that the semiflow $\Pi$ admits a compact
invariant set $\mathcal I$ (i.e. $\Pi(t)\mathcal I=\mathcal I$ for all $t\geq0$).
We do not make any structure assumption on the nonlinearity $f(x,u)$ and therefore
 we do not assume that $\mathcal I$ is
the global attractor of equation (\ref{equation1}). Our aim is to prove that
$\mathcal I$ has finite Hausdorff and fractal dimension and to give an
explicit estimate of its dimension.

When $\Omega$ is a bounded
domain and $f(x,u)$ satisfies suitable dissipativeness conditions,
the existence of a finite dimensional compact global
attractor for (\ref{equation1})  is  a classical achievement (see e.g. \cite{Lady, Tem}
and the references therein).

When $\Omega$ is unbounded, new difficulties arise due to the lack
of compactness of the Sobolev embeddings. These difficulties can
be overcome in several ways: by exploiting the \emph{finite speed of propagation property}
(e.g. in \cite{F}), by introducing \emph{weighted or uniform spaces} (see
e.g. \cite{Z}), by developing suitable \emph{tail-estimates} (see
e.g. \cite{PR1}).

Concerning the finite dimensionality of the attractor, in the
\emph{unbounded domain case} very few results are available. In
\cite{Z} Zelik proved finite dimensionality of attractors in the
context of uniform spaces, assuming that $\beta(x)$ is constant
and $f(x,u)$ is independent of $x$ and satisfies $f(u)u\leq 0$,
$f'(u)\leq L$ for all $u\in\R$. The technique exploited by Zelik
seems not to give explicit bounds for the dimension of the
attractor. In \cite{KaSta}, Karachalios and Stavrakakis considered
an equation of the form
\begin{equation} u_{tt}+\alpha
u_t+\beta(x)u-g(x)^{-1}\Delta u=f(u)+h(x),
\end{equation}
 where $g(\cdot)$ is a
positive function belonging to $L^\infty\cap L^{3/2}$. In this
case the weight $g(x)^{-1}$ ``forces" the operator
$-g(x)^{-1}\Delta$ to have compact resolvent: the result then is
achieved by exploiting directly the technique of \emph{volume
tracking} developed by Temam  and other authors  for bounded
domains (see \cite{Tem}).

In this paper we do not make any structure assumption on the
nonlinearity $f(x,u)$. Our only assumption is that $\partial_u
f(x,0)$ is non negative and belongs to $L^r(\Omega)$ for some
$r>3$. The positivity of $\partial_u f(x,0)$ is not a real
restriction, because its negative part can be absorbed in
$\beta(x)$. Under this assumption, we shall prove that $\mathcal
I$ has finite Hausdorff and fractal dimension in the energy space
$H^1_0(\Omega)\times L^2(\Omega)$. Also, we give an explicit
estimate of the dimension of $\mathcal I$, in terms of the main
parameters involved in the equation and of the quantity
$\sup\{\|(u,v)\|\mid (u,v)\in\mathcal I\}$. In order to achieve
our result, we shall exploit the technique of \emph{volume
tracking}, as expounded in \cite{Tem}. However, we cannot apply
directly the arguments of \cite{Tem}, since the operator
$-\Delta+\beta(x)$ does not have compact resolvent. Indeed, in the
\emph{bounded domain case} (resp. in the \emph{weighted Laplacian
case} considered by Karachalios and Stavrakakis) the key point is
that
\begin{equation}\label{intro2}
\frac1d\sum_{j=1}^d\lambda_j^{-1}\to
0\quad\text{as $d\to\infty$,}
\end{equation}
where $(\lambda_j)_{j\in\N}$ is the sequence of the eigenvalues of
$-\Delta$ (resp. of $-g(x)^{-1}\Delta$). In general the operator
$-\Delta+\beta(x)$, when $\Omega$ is unbounded, does not satisfy
such property, since it possesses a nontrivial essential spectrum
and its eigenvalues below the bottom of the essential spectrum are
finite or form a sequence which accumulate to the bottom the
essential spectrum. Yet, a more accurate analysis shows that the
numbers $\lambda_j$ in (\ref{intro2}) can be replaced by
$\check\lambda_j$, where $(\check\lambda_j)_{j\in\N}$ is the
sequence of the eigenvalues of the following \emph{weighted
eigenvalue problem}:
\begin{equation}\label{intro3}
-\Delta \phi+\beta(x) \phi=\check\lambda\partial_u f(x, \bar
u(x))^2\phi,
\end{equation}
where $\bar U=(\bar u,\bar v)\in\mathcal I$. It turns out that
(\ref{intro3}) has a pure point spectrum. Moreover, thanks to the
Cwickel-Lieb-Rozenblum inequality, it is possible to determine the
asymptotics of the sequence $(\check\lambda_j)_{j\in\N}$
independently of $\bar U\in\mathcal I$, and the result will
follow.

The paper is organized as follows. In Section 2 we introduce
notations, we state the main assumptions and we collect some
preliminaries about the semiflow generated by equation
(\ref{equation1}).  In Section 3 we recall the definition of
Hausdorff and fractal dimension and we prove that any compact
invariant set $\mathcal I$ of $\Pi$ has finite Hausdorff and
fractal dimension in $H^1_0(\Omega)\times L^2(\Omega)$. In Section
4
 we specialize our result to the case of  dissipative equations and we
show that the dimensions of the attractors of (\ref{equation1})
remain bounded as $\alpha\to\infty$.


\section {Notation, preliminaries and remarks}

Let $\sigma\geq 1$. We denote by $L^\sigma_{\rm u}(\R^N)$ the set of measurable functions $\omega\colon \R^N\to\R$ such that
\begin{equation*}
|\omega|_{L^\sigma_{\rm u}}:=\sup_{y\in\R^N}\left(\int_{B(y)}|\omega(x)|^\sigma\,dx\right)^{1/\sigma}<\infty,
\end{equation*}
where, for $y\in\R^N$, $B(y)$ is the open unit cube in $\R^N$ centered at $y$.

In this paper we assume throughout that $N=3$, and we fix an open (possibly unbounded) set $\Omega\subset\R^3$.

\begin{Prop}\label{prop1}
Let $\sigma>3/2$ and let $\omega\in L^\sigma_{\rm u}(\R^3)$. Set $\rho:= 3/2\sigma$. Then, for every $\epsilon>0$ and for every $u\in H^1_0(\Omega)$,
\begin{equation}
\int_{\Omega}|\omega(x)| |u(x)|^2\,dx\leq  |\omega|_{L^\sigma_{\rm u}}
\left(\rho\epsilon M_B^{2}|u|_{H^1}^{2}+(1-\rho)\epsilon^{-\rho/(1-\rho)}|u|_{L^2}^{2}\right),
\end{equation}
where $M_B$ the constant of the Sobolev embedding $H^1(B)\subset  L^6(B)$ and $B$ is the open unit
cube in $\R^3$.
Moreover, for every $u\in H^1_0(\Omega)$,
\begin{equation}
\int_{\Omega}|\omega(x)| |u(x)|^2\,dx\leq M_B^{2\rho} |\omega|_{L^\sigma_{\rm u}}|u|_{H^1}^{2\rho}|u|_{L^2}^{2(1-\rho)}.
\end{equation}
\end{Prop}
\begin{proof}See the proof of Lemma 3.3 in \cite{PR2}.\end{proof}

Let $\beta\in L^\sigma_{\rm u}(\R^3)$, with $\sigma>3/2$. Let us consider the following bilinear form defined on the space $H^1_0(\Omega)$:
\begin{equation}
a(u,v):=\int_\Omega\nabla u(x)\cdot\nabla v(x)\,dx+\int_\Omega\beta(x)u(x)v(x)\,dx,\quad u,v\in H^1_0(\Omega)
\end{equation}
Our first assumption is the following:
\begin{Hyp}\label{hyp1}
There exists $\lambda_1>0$ such that
\begin{equation}
\int_\Omega |\nabla u(x)|^2\,dx+\int_\Omega\beta(x)|u(x)|^2\,dx\geq \lambda_1 |u|^2_{L^2},\quad u\in H^1_0(\Omega).\label{f10}
\end{equation}
\end{Hyp}

\begin{Rem} Conditions on $\beta(x)$ under which Hypothesis \ref{hyp1} is satisfied are expounded e.g. in \cite{AB1,AB2}.
\end{Rem}

As a consequence of (\ref{f10}) and Proposition \ref{prop1}, we have:
\begin{Prop}\label{prop2} There exist two positive constants $\lambda_0$ and $\Lambda_0$ such that
\begin{equation}
\lambda_0|u|^2_{H^1}\leq\int_\Omega |\nabla u(x)|^2\,dx+\int_\Omega\beta(x)|u(x)|^2\,dx\leq \Lambda_0 |u|^2_{H^1},\quad u\in H^1_0(\Omega).
\end{equation}
The constants $\lambda_0$ and $\Lambda_0$ can be computed explicitly in terms of $\lambda_1$, $M_B$ and $|\beta|_{L^\sigma_{\rm u}}$.
\end{Prop}
\begin{proof}Cf Lemma 4.2 in \cite{PR1}\end{proof}
It follows from Proposition \ref{prop2} that the bilinear form $a(\cdot,\cdot)$ defines a scalar product in $H^1_0(\Omega)$,
equivalent to the standard one.

\begin{Notat} From now on, we set  $\langle\cdot,\cdot\rangle_{H^1_0}:=a(\cdot,\cdot)$ and we denote by $\|\cdot\|_{H^1_0}$ the norm associated with
$\langle \cdot,\cdot\rangle_{H^1_0}$. Also, we shall use the notation $\|\cdot\|_{L^p}$ to denote the $L^p$-norm in $L^p(\Omega)$,  $1\leq p\leq\infty$. \end{Notat}

Let $\mathbf A$ be the self-adjoint operator on $L^2(\Omega)$
defined by the differential operator $u\mapsto\beta(x) u-\Delta
u$.

Then $\mathbf A$ generates a
family $X^\kappa$, $\kappa\in\R$, of fractional power spaces with
$X^{-\kappa}$ being the dual of $X^\kappa$ for
$\kappa\in]0,+\infty[$. For $\kappa\in]0,+\infty[$, the space
$X^\kappa$ is a Hilbert space with respect to the scalar product
\begin{equation*}
\langle u,v\rangle_{X^\kappa}:=\langle{\mathbf A}^\kappa
u,{\mathbf A}^\kappa v\rangle_{L^2},\quad u,v\in X^\kappa.
\end{equation*}
Also, the space $X^{-\kappa}$ is a Hilbert space with respect to
the scalar product $\langle \cdot,\cdot\rangle_{X^{-\kappa}}$ dual
to the scalar product  $\langle \cdot,\cdot\rangle_{X^{\kappa}}$,
i.e.
\begin{equation*}
\langle u',v'\rangle_{X^{-\kappa}}=\langle R^{-1}_\kappa
u',R^{-1}_\kappa v'\rangle_{X^\kappa},\quad u,v\in X^{-\kappa},
\end{equation*}
where $R_\kappa\colon X^\kappa\to X^{-\kappa}$ is the Riesz
isomorphism $u\mapsto\langle\cdot,u\rangle_{X^\kappa}$.

We make the following assumption:

\begin{Hyp}\label{hyp4}
The open set  $\Omega$ is a uniformly $C^2$
domain in the sense of Browder \cite[p. 36]{Brow}.
\end{Hyp}

As a consequence,  by elliptic
regularity we have that $D(-\Delta)=H^2(\Omega)\cap
H^1_0(\Omega)\subset L^\infty(\Omega)$. In this situation,
the assignment $u\mapsto
\beta(x) u$ defines a relatively bounded perturbation of $-\Delta$
and therefore $D(-\Delta+\beta(x))=H^2(\Omega)\cap H^1_0(\Omega)$.
It follows that $X^\kappa\subset L^\infty(\Omega)$ for
$\kappa>3/4$ (see \cite[Th. 1.6.1]{He}).

 We write
\begin{equation*}
H_\kappa=X^{\kappa/2},\quad\kappa\in\R.
\end{equation*}
Note that $H_0=L^2(\Omega)$, $H_1=H^1_0(\Omega)$,
$H_{-1}=H^{-1}(\Omega)$ and $H_2=H^2(\Omega)\cap H^1_0(\Omega)$.

For $\kappa\in\R$ the operator $\mathbf A$ induces a self-adjoint
operator $\mathbf A_\kappa\colon H_{\kappa+2}\to H_{\kappa}$. In
particular $\mathbf A=\mathbf A_0$. Moreover,
\begin{equation*}
\langle u,v\rangle_{H^1_0}=\langle\mathbf A_0
u,v\rangle_{L^2},\quad u\in D(\mathbf A_0),\,v\in H^1_0(\Omega).
\end{equation*}

For  $\kappa\in\R$ set $Z_\kappa:=H_{\kappa+1}\times H_{\kappa}$.
For $\alpha>0$ define the linear operator $\mathbf
B_{\kappa}\colon Z_{\kappa+1}\to Z_{\kappa}$ by
\begin{equation*}
\mathbf B_{\kappa}(u,v):=(v,-(\alpha v+\mathbf
A_\kappa u)),\quad (u,v)\in Z_{\kappa+1}.
\end{equation*}
It follows that $\mathbf B_{\kappa}$ is $m$-dissipative
on $Z_{\kappa}$ (cf the proof of Prop. 3.6 in \cite{PR2}).
Therefore, by the Hille-Yosida-Phillips theorem (see e.g.
\cite{CH}), $\mathbf B_{\kappa}$ is the infinitesimal
generator of a $C^0$-semigroup $\mathbf T_{\kappa}(t)$,
$t\in[0,+\infty[$, on $Z_{\kappa}$.

Given  a function $g\colon \Omega\times \R\to \R$, we denote by
$\hat g$  the Nemitski operator which associates with every
function $u\colon \Omega\to \R$ the function $\hat g(u)\colon
\Omega\to \R$ defined by
 $$
  \hat g(u)(x)= g(x,u(x)),\quad x\in \Omega.
 $$
If $I\subset \R$,  $X$ is a normed spaces
and if $u\colon I\to X$ is a function which is differentiable as a
function into $X$ then we denote its $X$-valued derivative by
$(\partial_t \mid X)\,u$. Similarly, if $X$ is a Banach space and
$u\colon I\to X$ is integrable as a function into $X$, then we
denote its $X$-valued integral by $\int _I u(t)\,(dt\mid X)$. If
$X$ and $Y$ are Banach spaces, we denote by $\mathcal L(X,Y)$ the
space of bounded linear operators from $X$ to $Y$. If $X=Y$ we
write just $\mathcal L(X)$.

\begin{Hyp}\label{hyp2}\
\begin{enumerate}
\item $f\colon\Omega\times\R\to\R$ is such that, for every $u\in\R$, $f(\cdot,u)$ is measurable and $f(\cdot,0)\in L^2(\Omega)$;
\item for a.e. $x\in\Omega$, $f(x,\cdot)$ is of class $C^2$, $\partial_{u}f(\cdot,0)\in L^\infty(\Omega)$ and there exists a constants $C\geq0$
such that $$ |\partial_{uu} f(x,u))|\leq C(1+|u|), \quad
(x,u)\in\Omega\times\R. $$
\end{enumerate}\end{Hyp}

The main properties of the Nemitski operator associated with $f$
are collected in the following Proposition, whose proof is left to
the reader.
\begin{Prop}\label{newprop1} Assume Hypothesis \ref{hyp2}. Then
$\hat f\colon H^1_0(\Omega)\to L^2(\Omega)$ is continuously
differentiable, $D\hat f(u)[v](x)=\partial_u f(x,u(x))v(x)$ for
$u$, $v\in H^1_0(\Omega)$, and there exists a positive constant
$\tilde C>0$ such  that the following estimates hold:
\begin{equation}\label{n1}
\|\hat f(u)\|_{L^2}\leq\tilde C(1+\|u\|_{H^1_0}^3),\quad u\in
H^1_0(\Omega)
\end{equation}
\begin{equation}\label{n2}
\|D\hat f (u)\|_{{\mathcal L}(H^1_0,L^2)}\leq \tilde
C(1+\|u\|_{H^1_0}^2),\quad u\in H^1_0(\Omega)
\end{equation}
\begin{multline}\label{n3}
\|D\hat f (u_1)-D\hat f (u_2)\|_{{\mathcal L}(H^1_0,L^2)}\leq
\tilde C(1+\|u_1\|_{H^1_0}+\|u_2\|_{H^1_0})\|u_1-u_2\|_{H^1_0},\\
u_1,u_2\in H^1_0(\Omega).
\end{multline}
If $u\in H^1_0(\Omega)$ and $v\in L^2(\Omega)$, then
$\widehat{\partial_u f}(u)\cdot v\in H^{-1}(\Omega)$ and the
following estimates hold:
\begin{equation}\label{n4}
\|\widehat{\partial_u f}(u)\|_{\mathcal L(L^2,H^{-1})}\leq\tilde
C(1+\|u\|_{H^1_0}^2),\quad u\in H^1_0(\Omega)
\end{equation}
\begin{multline}\label{n5}
\|\widehat{\partial_u f}(u_1)-\widehat{\partial_u
f}(u_2)\|_{\mathcal L(L^2,H^{-1})}\leq\tilde
C(1+\|u_1\|_{H^1_0}+\|u_2\|_{H^1_0})\|u_1-u_2\|_{H^1_0},\\
u_1,u_2\in H^1_0(\Omega).
\end{multline}
\qed
\end{Prop}

We consider the following semi-linear damped wave equation:
\begin{equation}\label{eq1}
\begin{aligned}
u_{tt}+\alpha u_t+\beta(x)u-\Delta u
&=f(x,u),&&(t,x)\in[0,+\infty[\times\Omega,\\
u&=0,&&(t,x)\in[0,+\infty[\times\partial\Omega
\end{aligned}\end{equation}
with Cauchy data $u(0)=u_0$, $u_t(0)=v_0$.

We recall the following classical result (see e.g. Theorem II.1.3
in \cite{Gold}):

\begin{Theorem}\label{goldstein}
Let $X$ be a Banach space and let $B\colon D(B)\subset X\to X$ be
the infinitesimal generator of a $C^0$-semigroup of linear
operators $T(t)$, $t\in\R_+$. Consider the abstract Cauchy problem
\begin{equation}\label{goldstein-cauchy-1}
\begin{cases}\dot u=B u(t)+f(t),&t\in\R_+\\
u(0)=u_0&
\end{cases}
\end{equation}
Assume that $u_0\in D(B)$ and that either
\begin{enumerate}
\item $f\in C(\R_+,X)$ takes values in $D(B)$ and $Bf\in
C(\R_+,X)$, or
\item $f\in C^1(\R_+,X)$.
\end{enumerate}
Then (\ref{goldstein-cauchy-1}) has a unique solution $u\in
C^1(\R_+)$ with values in $D(B)$. The solution is given by
\begin{equation}\label{goldstein-cauchy-2}
u(t)=T(t)u_0+\int_0^t T(t-s)f(s)\,ds.
\end{equation}\qed
\end{Theorem}

Using Theorem \ref{goldstein}, we rewrite equation (\ref{eq1}) as
an integral evolution equation in the space
$Z_0=H^1_0(\Omega)\times L^2(\Omega)$, namely
\begin{equation}\label{eq2}
(u(t),v(t))=\mathbf T_{0}(t)(u_0,v_0)+\int_0^t \mathbf
T_{0}(t-p)(0,\hat f(u(p)))\,(dp\mid Z_0).
\end{equation}
Equation (\ref{eq2}) is called the {\em mild formulation} of
(\ref{eq1}) and  solutions of (\ref{eq2}) are called {\em mild
solutions} of (\ref{eq1}). Note that by Proposition \ref{prop1}
the nonlinear operator $(u,v)\mapsto(0,\hat f(u))$ is Lipschitz
continuous from $Z_0$ into itself. A classical Picard iteration
argument shows that, if $(u_0,v_0)\in Z_0$, then (\ref{eq2})
possesses a unique continuous maximal solution
$(u(\cdot),v(\cdot))\colon[0,t_{\mathrm{max}}[\to Z_0$ (see Theor.
4.3.4 and Prop. 4.3.7 in \cite{CH}). We thus obtain a local
semiflow on $Z_0$, which we denote by $\Pi(t)U_0$, $U_0=(u_0,v_0)\in Z_0$, $t\in[0, t_{\mathrm{max}}(U_0)[$. Notice that the solution $(u(\cdot),v(\cdot))$
of (\ref{eq2}) also satisfies
\begin{equation}\label{eq2-bis}
(u(t),v(t))=\mathbf T_{-1}(t)(u_0,v_0)+\int_0^t \mathbf
T_{-1}(t-p)(0,\hat f(u(p)))\,(dp\mid Z_{-1}).
\end{equation}
Therefore, it follows from Theorem \ref{goldstein} that
$(u(\cdot),v(\cdot))$ is continuously differentiable into $Z_{-1}$
and
\begin{equation}\label{eq3}
(\partial_t\mid Z_{-1})(u(t),v(t))=\mathbf
B_{-1}(u(t),v(t))+(0,\hat f(u(t))).
\end{equation}
In particular, one has
\begin{equation}\label{eq4}\begin{cases}
(\partial_t\mid H_0)u(t)=v(t)\\ (\partial_t\mid
H_{-1})v(t)=-\alpha v(t)-\mathbf A_{-1}u(t)+\hat f(u(t))
\end{cases}
\end{equation}
\begin{Def}\label{def1}
A function $(u(\cdot), v(\cdot))\colon\R\to Z_0$ is called a {\em
full solution} of the semiflow $\Pi$ generated by (\ref{eq2}) iff, for every $s$, $t\in\R$, with
$s\leq t$, one has
\begin{equation*}
(u(t),v(t))=\Pi(t-s)(u(s),v(s))
\end{equation*}
\end{Def}

\begin{Def}\label{newdef1} A subset $\mathcal I$ of $Z_0$ is called {\em
invariant} for the semiflow generated by (\ref{eq2}) if for every
$(u_0,v_0)\in\mathcal I$ there exists a full solution
$(u(\cdot),v(\cdot))$ of (\ref{eq2}) with $(u(0),v(0))=(u_0,v_0)$
and $(u(t),v(t))\in\mathcal I$ for all $t\in\R$.\end{Def}

From now on we assume that $\mathcal I\subset Z_0$ is a compact invariant subset of the semiflow $\Pi$.

\begin{Notat}If $\mathcal B$ is  a Banach space such that $\mathcal I\subset \mathcal B$, we define
\begin{equation}
|\mathcal I |_{\mathcal B}:=\max\{ \|u\|_{\mathcal B}\mid u\in\mathcal I\}.
\end{equation}
\end{Notat}

We recall the following result:

\begin{Theorem}[\bf{cf Corollaries 2.10 and 2.13 in \cite{P1}}]\label{newth1}
Assume that Hypotheses \ref{hyp1}, \ref{hyp4} and \ref{hyp2} are
satisfied. Let $\mathcal I\subset Z_0$ be a compact invariant set
of the semiflow generated by (\ref{eq2}). Then $\mathcal I$ is a
bounded subset of $Z_1$. Moreover, $|\mathcal I|_{Z_1}$ can be
explicitly estimated in terms of $|\mathcal I|_{Z_0}$ and of the
constants in Hypotheses \ref{hyp1} and \ref{hyp4}.
\end{Theorem}

Let $\bar U_0=(\bar u_0,\bar v_0)\in\mathcal I$, and let $\bar U(t)=(\bar u(t),\bar v(t))$, $t\in\R$, be the full bounded solution through $\bar U_0$.
Given $H_0=(h_0,k_0)\in Z_0$, let us denote by $\mathcal U(\bar U_0; t) H_0$ the mild solution of
\begin{equation}\label{na-eq1}
(h(t),k(t))=\mathbf T_{0}(t)(h_0,k_0)+\int_0^t \mathbf
T_{0}(t-p)(0,\widehat{\partial_u  f}(\bar u(p))h(p))\,(dp\mid Z_0).
\end{equation}

Notice that $\mathcal U(\bar U_0; t)$ coincides with the restriction to $Z_0$ of the \emph{
evolution family} $\mathbf U_{-1}(t,s)$ generated in $Z_{-1}$ by the family $\mathbf B_{-1}+\mathbf C_{-1}(t)$, $t\in\R$, where $\mathbf C_{-1}(t)(h,k):=(0,\widehat{\partial_u  f}(\bar u(t))h)$
(see \cite{P1} and \cite{K}).

A standard computation using Gronwall's inequality and Proposition \ref{newprop1} leads to the following:

\begin{Prop}\label{newprop2} For every $t\geq0$,
\begin{equation}
\sup_{\bar U_0\in\mathcal I}\|\mathcal U(\bar U_0;t)\|_{\mathcal L(Z_0,Z_0)}<+\infty,
\end{equation}
and
\begin{equation}
\lim_{\epsilon\to 0}\sup_{\substack{\bar U_1,\bar U_2\in\mathcal I\\0<\|\bar U_1-\bar U_2\|_{Z_0}<\epsilon}}
\frac{\|\Pi(t)(\bar U_2)-\Pi(t)(\bar U_1)-\mathcal U(\bar U_1;t)(\bar U_2-\bar U_1)\|_{Z_0}}{\|\bar U_2-\bar U_1\|_{Z_0}}=0,
\end{equation}
where $\bar Ui=(\bar u_i,\bar v_i)$, $i=1,2,3$.\qed
\end{Prop}


\section{Dimension of invariant sets}

Let $\mathcal X$ be a complete metric space and let $\mathcal
K\subset\mathcal X$ be a compact set. For $d\in \R^+$ and
$\epsilon>0$ one defines
\begin{equation}
\mu_{H}(\mathcal K,d,\epsilon):=\inf \left\{\sum_{i\in I}r_i^d\mid
\mathcal K\subset \bigcup_{i\in I}B(x_i,r_i),\, r_i\leq\epsilon
\right\},
\end{equation}
where the infimum is taken over all the finite coverings of
$\mathcal K$ with balls of radius $r_i\leq\epsilon$. Observe that
$\mu_{H}(\mathcal K,d,\epsilon)$ is a non increasing function of
$\epsilon$ and $d$. The $d$-dimensional Hausdorff measure of
$\mathcal K$ is by definition
\begin{equation}
\mu_H(\mathcal K,d):=\lim_{\epsilon\to 0}\mu_{H}(\mathcal
K,d,\epsilon)= \sup_{\epsilon>0}\mu_{H}(\mathcal K,d,\epsilon).
\end{equation}
One has:
\begin{enumerate}
\item $\mu_H(\mathcal K,d)\in[0,+\infty]$;
\item if $\mu_H(\mathcal K,\bar d)<\infty$, then $\mu_H(\mathcal
K,d)=0$ for all $d>\bar d$; \item if $\mu_H(\mathcal K,\bar d)>0$,
then $\mu_H(\mathcal K,d)=+\infty$ for all $d<\bar d$.
\end{enumerate}
The Hausdorff dimension of $\mathcal K$ is the smallest $d$ for
which $\mu_H(\mathcal K,d)$ is finite, i.e.
\begin{equation}
{\rm dim}_{H}(\mathcal K):=\inf\{d>0\mid \mu_H(\mathcal K,d)=0\}.
\end{equation}

Now let $n_{\mathcal K}(\epsilon)$, $\epsilon>0$, denote the
minimum number of balls of $\mathcal X$ of radius $\epsilon$ which
is necessary to cover $\mathcal K$. The fractal dimension of
$\mathcal K$ is the number

\begin{equation}
{\rm dim}_{F}(\mathcal K):=\limsup_{\epsilon\to 0}\frac{\log
n_{\mathcal K}(\epsilon)}{\log 1/\epsilon}.
\end{equation}

There is a well developed technique to estimate the Hausdorff
dimension of an invariant set of a map or a semigroup. We refer
the reader e.g. to \cite{Tem} and \cite{Lady}. The geometric idea
consists in tracking the evolution of a $d$-dimensional volume
under the action of the linearization of the semigroup along
solutions lying in the invariant set. One looks then for the
smallest $d$ for which any $d$-dimensional volume contracts
asymptotically as $t\to\infty$.

We fix $\delta\in\R$ and we introduce a \emph{change of coordinates} in the space $Z_\kappa$, $\kappa\in\R$, by
\begin{equation*}
R_\delta\colon Z_{\kappa}\to Z_{\kappa}, \quad (u,v)\mapsto (u,v+\delta u).
\end{equation*}
The constant $\delta$ is to be fixed later. Clearly the transformation $R_\delta$ is linear, bounded and invertible, with inverse $R_\delta^{-1}=R_{-\delta}$.
We define the semiflow $$\Pi_\delta(t):=R_\delta\circ\Pi(t)\circ R_{-\delta}$$ and we set $\mathcal I_\delta:=R_\delta \mathcal I$. Then $\mathcal I_\delta$ is a compact invariant set of $\Pi_\delta$,
it is bounded in $Z_1$, and $\dim \mathcal I_\delta=\dim\mathcal I$. For $\tilde U_0\in\mathcal I_\delta$ and $t\geq 0$ we set
$$
\mathcal U_\delta(\tilde U_0;t):=R_\delta \circ \mathcal U(R_{-\delta}\tilde U_0;t)\circ R_{-\delta}.
$$
Then the conclusions of Proposition \ref{newprop2} hold with $\Pi(t)$, $\mathcal I$ and $\mathcal U(\bar U;t)$ replaced by $\Pi_\delta(t)$, $\mathcal I_\delta$ and $\mathcal U_\delta(\tilde U;t)$.

Let $\tilde U_0=(\tilde u_0,\tilde v_0)\in\mathcal I_\delta$ and
let $\tilde U (t)=(\tilde u(t),\tilde v(t))=\Pi_\delta (t)\tilde
U_0$. Let $\Phi_{0,i}$, $i=1$, \dots, $d$, be linearly independent
elements of $Z_0$, $\Phi_{0,i}=(\phi_{0,i},\psi_{0,i})$. Set
$\Phi_i(t):=\mathcal U_\delta(\tilde U_0;t)\Phi_{0,i}$. We denote
by $G(t)$
 the square of  the $d$-dimensional volume delimited by $\Phi_1(t)$, \dots, $\Phi_\delta(t)$, that is
 \begin{equation}
 G(t):=\|\Phi_1(t)\wedge\dots\wedge \Phi_d(t)\|_{\wedge^d Z_0}^2=\det(\langle \Phi_i(t),\Phi_j(t)\rangle_{Z_0})_{ij}.
 \end{equation}
We need to find a differential equation satisfied by $G(t)$.
\begin{Lemma}\label{derivative}
Let $i$ and $j\in\{1,\dots,d\}$ be fixed. Then the function $t\mapsto \langle \Phi_i(t),\Phi_j(t)\rangle_{Z_0}$ is continuously differentiable, and
\begin{multline}
\frac d{dt} \langle \Phi_i,\Phi_j\rangle_{Z_0}=-2\delta\langle
\phi_i,\phi_j\rangle_{H^1_0}-2(\alpha-\delta)\langle\psi_i,\psi_j\rangle_{L^2}\\
+\delta(\alpha-\delta)(\langle\phi_i,\psi_j\rangle_{L^2}+\langle\psi_i,\phi_j\rangle_{L^2})+
(\langle\widehat{\partial_u f}(\tilde
u(t))\phi_i,\psi_j\rangle_{L^2}+\langle\psi_i,\widehat{\partial_u
f}(\tilde u(t))\phi_j\rangle_{L^2}).
\end{multline}
\end{Lemma}
\begin{proof}
 First set
$\bar U_0:=R_{-\delta}\tilde U_0$, $\bar U(t):=R_{-\delta}\tilde U(t)$, $\Theta_{0,l}=(\theta_{0,l},\chi_{0,l}):=R_{-\delta}\Phi_{0,l}$, $l=i,j$,
and $\Theta_{l}(t)=(\theta_{l}(t),\chi_{l}(t)):=R_{-\delta}\Phi_{l}(t)$, $l=i,j$. Notice that $\Theta_l(t)=\mathcal U(\bar U_0;t)\Theta_{0,l}$, $l=i,j$.
It follows that $\langle \Phi_i(t),\Phi_j(t)\rangle_{Z_0}=\langle R_\delta \Theta_i(t),R_\delta\Theta_j(t)\rangle_{Z_0}$.
Now we shall apply Theorem 2.6 in \cite{PR2}.
Set:
\begin{itemize}
\item $Z:=Z_0\oplus Z_0$;
\item $T(t):=\mathbf T_0(t)\oplus\mathbf T_0(t)$;
\item $B:=\mathbf B_0\oplus\mathbf B_0$;
\item $g(s)=(0, \widehat{\partial_u f}(\tilde u(t))\theta_i(t))\oplus(0,\widehat{\partial_u f}(\tilde u(t))\theta_j(t)) $;
\item $z(t)=\Theta_i(t)\oplus\Theta_j(t)$;
\item $V(U_1,U_2):=\langle R_\delta U_1,R_\delta U_2\rangle_{Z_0}$
\end{itemize}
A standard computation shows that $V$ is Fr\'echet differentiable in $Z$; moreover, for $U_i\oplus U_j \in D(B)$ and $H_i\oplus H_j\in Z$,
\begin{multline*}
DV(U_i\oplus U_j)[B (U_i\oplus U_j)+H_i\oplus H_j] =\langle
v_i+h_i,u_j\rangle_{H^1_0}+\delta\langle v_i+h_i,\delta u_j+v_j
\rangle_{L^2}\\+\langle-\alpha v_i+k_i,\delta
u_j+v_j\rangle_{L^2}-\langle \mathbf A_0 u_i,\delta u_j+v_j
\rangle_{L^2} +\langle
u_i,v_j+h_j\rangle_{H^1_0}\\+\delta\langle\delta u_i+v_i,v_j+h_j
\rangle_{L^2}+\langle\delta u_i+v_i,-\alpha v_j+k_j
\rangle_{L^2}-\langle\delta u_i+v_i,\mathbf A_0 u_j \rangle_{L^2}
\\=-2\delta\langle u_i,u_j\rangle_{H^1_0} +(\langle
h_i,u_j\rangle_{H^1_0}+\langle u_i,h_j\rangle_{H^1_0})+(\langle
k_i,\delta u_j+v_j\rangle_{L^2}+\langle\delta
u_i+v_i,k_j\rangle_{L^2})\\ +(\langle \delta(v_i+h_i)-\alpha
v_i,\delta u_j+v_j \rangle_{L^2}+\langle \delta
u_i+v_i,\delta(v_j+h_j)-\alpha v_j\rangle_{L^2})
\end{multline*}
where $U_l=(u_l,v_l)$ and $H_l=(h_l,k_l)$, $l=i,j$. It follows from Theorem 2.6 in \cite{PR2} that
\begin{multline*}
\frac d{dt} \langle \Phi_i,\Phi_j\rangle_{Z_0}=\frac d{dt} V(\Theta_i,\Theta_j)=-2\delta\langle \theta_i,\theta_j\rangle_{H^1_0}+(\langle(\delta-\alpha)\chi_i,\delta\theta_j+\chi_j\rangle_{L^2}\\
+\langle \delta\theta_i+\chi_i,(\delta-\alpha)\chi_j\rangle_{L^2}
+(\langle \widehat{\partial_u f}(\tilde u(t))\theta_i,\delta\theta_j+\chi_j\rangle_{L^2}+\langle\delta\theta_i+\chi_i,\widehat{\partial_u f}(\tilde u(t))\theta_j\rangle_{L^2}\\
=-2\delta\langle \phi_i,\phi_j\rangle_{H^1_0}-2(\alpha-\delta)\langle\psi_i,\psi_j\rangle_{L^2}
+\delta(\alpha-\delta)(\langle\phi_i,\psi_j\rangle_{L^2}+\langle\psi_i,\phi_j\rangle_{L^2})\\
+(\langle\widehat{\partial_u f}(\tilde u(t))\phi_i,\psi_j\rangle_{L^2}+\langle\psi_i,\widehat{\partial_u f}(\tilde u(t))\phi_j\rangle_{L^2})
\end{multline*}
and the proof is completed.
\end{proof}

Let $\tilde U=(\tilde u,\tilde v)\in\mathcal I_\delta$ and let
$\Sigma_d$ be a $d$-dimensional subspace of $Z_0$. On $\Sigma_d$
we define a self-adjoint operator $\mathbf B_{\tilde
U,\Sigma_d,\delta}$ by
\begin{multline}
\langle\mathbf B_{\tilde U,
\Sigma_d,\delta}(u,v),(\xi,\eta)\rangle_{Z_0}:=-2\delta\langle
u,\xi\rangle_{H^1_0}-2(\alpha-\delta)\langle v,\eta\rangle_{L^2}\\
+\delta(\alpha-\delta)(\langle u,\eta \rangle_{L^2}+\langle
v,\xi\rangle_{L^2})+ (\langle\widehat{\partial_u f}(\tilde u)
u,\eta\rangle_{L^2}+\langle v,\widehat{\partial_u f}(\tilde
u)\xi\rangle_{L^2}),
\end{multline}
for $(u,v)$ and $(\xi,\eta)\in\Sigma_d$.

Now let $\tilde U_0$, $\tilde U(t)$, $\Phi_{0,i}$ and
$\Phi_{i}(t)$, $i=1$, \dots, $d$, and $G(t)$ be as above. We set
 $\Sigma_d(t):={\mathrm{span}}(\Phi_1(t),\dots,\Phi_d(t))$ and we define a $(d\times d)$- matrix $(b_{il}(t))_{il}$
 such that $$\mathbf B_{\tilde U(t),\Sigma_d(t),\delta}\Phi_i(t)=\sum_{l=1}^db_{il}(t)\Phi_l(t).$$
It follows from Lemma \ref{derivative} that
\begin{equation}
\frac d{dt} \langle
\Phi_i(t),\Phi_j(t)\rangle_{Z_0}=\langle\mathbf B_{\tilde
U(t),\Sigma_d(t),\delta}\Phi_i(t),\Phi_j(t)\rangle_{Z_0}=\sum_{l=1}^db_{il}(t)\langle\Phi_l(t),\Phi_j(t)\rangle_{Z_0}.
\end{equation}
A straightforward computation now shows that
\begin{equation}
\frac d{dt} G(t)=\left(\sum_{i=1}^d
b_{ii}(t)\right)G(t)=\mathrm{Tr} (\mathbf B_{\tilde
U(t),\Sigma_d(t),\delta}) G(t).
\end{equation}
Therefore we get:
\begin{equation}
\|\Phi_1(t)\wedge\dots\wedge \Phi_d(t)\|_{\wedge^d
Z_0}^2=\|\Phi_{0,1}\wedge\dots\wedge \Phi_{0,d}\|_{\wedge^d Z_0}^2
\exp\int_0^t \mathrm{Tr} (\mathbf B_{\tilde
U(s),\Sigma_d(s),\delta}) \,ds.
\end{equation}

For $j\in\N$, define the quantities
\begin{equation}\label{exponents}
p_j:=\sup\left\{\mathrm{Tr} (\mathbf B_{\tilde
U,\Sigma_j,\delta})\mid \Tilde U\in\mathcal I_\delta,\,
\Sigma_j\subset Z_0,\, \dim \Sigma_j=j \right\}.
\end{equation}
It follows from the results in \cite[Ch. V, pp 287--291]{Tem} that
if for some $d$ one has $p_d<0$ then the Hausdorff dimension of
$\mathcal I_\delta$ in $Z_0$ is finite and less than or equal to
$d$, and the fractal dimension of $\mathcal I_\delta$ in $Z_0$ is
finite and less than or equal to $d\max_{1\leq j\leq
d-1}(1+(p_j)_+/|p_d|)$. Therefore we must choose $\delta>0$ in
such a way that we can find $d$ such that $p_d<0$.

First we observe that, given an orthonormal basis $\check\Phi_1$,
\dots, $\check\Phi_d$ of $\Sigma_d$, then
\begin{multline}
\mathrm{Tr} (\mathbf B_{\tilde U,\Sigma_d,\delta})=\sum_{i=1}^d
\langle \mathbf B_{\tilde
U,\Sigma_d,\delta}\check\Phi_i,\check\Phi_i \rangle_{Z_0}\\
=\sum_{i=1}^d\left(-2\delta\|\check\phi_i\|_{H^1_0}^2-2(\alpha-\delta)\|\check\psi_i\|_{L^2}^2
+2\delta(\alpha-\delta)\langle\check\phi_i,\check\psi_i\rangle_{L^2}
+2\langle\widehat{\partial_u f}(\tilde
u)\check\phi_i,\check\psi_i\rangle_{L^2}\right),
\end{multline}
where $\check\Phi_i=(\check\phi_i,\check\psi_i)$, $i=1$, \dots,
$d$. Now, following the arguments of \cite{Shen}, we choose
$\delta:=\lambda_1\alpha/(\alpha^2+4\lambda_1)$. With this choice
of $\delta$, using Cauchy-Schwartz and Young's inequalities and
setting
\begin{equation}\label{def-nu}\nu_\alpha:=\frac{\lambda_1\alpha}{\sqrt{\alpha^2+4\lambda_1}(\alpha+\sqrt{\alpha^2+4\lambda_1})},\end{equation}
we get
\begin{equation}
\mathrm{Tr} (\mathbf B_{\tilde U,\Sigma_d,\delta}) \leq
-2\nu_\alpha d +
\sum_{i=1}^d\left(-\alpha\|\check\psi_i\|_{L^2}^2+2\langle\widehat{\partial_u
f}(\tilde u)\check\phi_i,\check\psi_i\rangle_{L^2}\right);
\end{equation}
using again Cauchy-Schwartz and Young's inequalities, we finally
obtain
\begin{equation}\label{traceinequality}
\mathrm{Tr} (\mathbf B_{\tilde U,\Sigma_d,\delta}) \leq
-2\nu_\alpha d +\frac1\alpha \sum_{i=1}^d\|\widehat{\partial_u
f}(\tilde u)\check\phi_i\|_{L^2}^2.
\end{equation}

\begin{Rem} Our choice of $\delta$, according to \cite{Shen}, is
better than the classical
$0<\delta\leq\min\{\alpha/4,\lambda_1/2\alpha\}$ (see e.g.
\cite{Tem}): indeed, when considering attractors of dissipative
wave equations, it yields dimensional bounds which are independent
of $\alpha$.
\end{Rem}

In order to prove finite dimensionality of $\mathcal I_\delta$, we
have now to find $d$ sufficiently large, so that the right hand
side of (\ref{traceinequality}) is negative, uniformly with
respect to $\tilde U$ and $\Sigma_\delta$. We introduce the
following fundamental Hypothesis:

\begin{Hyp}\label{hyp3}\
\begin{enumerate}
\item $\partial_u f(x,0)\geq 0$ for a.e. $x\in\Omega$;
\item there exists $r>3$ such that $\partial_u f(\cdot,0)\in
L^r(\Omega)$.
\end{enumerate}
Notice that property (1) is not really a restriction, since the
negative part of $\partial_u f(\cdot,0)$ can be absorbed by
$\beta(\cdot)$.\end{Hyp}

We observe that, by Hypotheses \ref{hyp2} and \ref{hyp3}, we have:
\begin{equation}
|\partial_u f(x,u)|\leq \partial_u f(x,0)+C(1+|u|)|u|,\quad
(x,u)\in\Omega\times\R.
\end{equation}

Take $\rho\in\mathcal S$ (the Schwartz class) with $\rho(x)>0$ for all $x\in\R^3$ and, for $\epsilon\geq 0$, define
\begin{equation}\label{pot}
W_{\tilde U}(x):=\partial_u f(x,0)+C(1+|\tilde
u|_{L^\infty})|\tilde u(x)|,\quad x\in\Omega,
\end{equation}
and
\begin{equation}
W_{\tilde U,\epsilon}(x):=W_{\tilde U}(x)+\epsilon\rho(x),\quad x\in\Omega.
\end{equation}
The reason for introducing the correction $\epsilon\rho(x)$ will be made clear later.
Notice that  $W_{\tilde U,\epsilon}(\cdot)\in L^r(\Omega)$ for $\epsilon\geq 0$ and $W_{\tilde U,\epsilon}>0$ for $x\in\Omega$ and $\epsilon>0$. Moreover,
\begin{equation}
\|\widehat{\partial_u f}(\tilde u)u\|_{L^2}^2\leq
\|W_{\tilde U} u\|_{L^2}^2\leq\|W_{\tilde U,\epsilon} u\|_{L^2}^2,\quad u\in H^1_0(\Omega).
\end{equation}
It follows from Lemma 4.5 in \cite{P2} that the assignment $u\mapsto W_{\tilde U,\epsilon}u$ defines a compact linear operator from $H^1_0(\Omega)$ to $L^2(\Omega)$.
Let us define the following operator $S_{\tilde U,\epsilon}\colon Z_0\to Z_0$:
\begin{equation}
S_{\tilde U,\epsilon}(u,v):=(0,W_{\tilde U,\epsilon}u),\quad U=(u,v)\in Z_0.
\end{equation}
Then $S_{\tilde U,\epsilon}$ is compact, and the same is true for its adjoint $S_{\tilde U,\epsilon}^*$. We have
\begin{equation}
\|W_{\tilde U,\epsilon} u\|_{L^2}^2=\langle S_{\tilde
U,\epsilon}U,S_{\tilde U,\epsilon}U\rangle_{Z_0}=\langle S_{\tilde
U,\epsilon}^*S_{\tilde U,\epsilon}U,U\rangle_{Z_0},\quad
U=(u,v)\in Z_0.
\end{equation}

The operator $S_{\tilde U,\epsilon}^*S_{\tilde U,\epsilon}$ is
compact, self-adjoint and non-negative. It follows that its
spectrum is
\begin{equation}
\sigma(S_{\tilde U,\epsilon}^*S_{\tilde U,\epsilon})=\{0\}\cup \{\mu_{\tilde U,\epsilon,j}\mid j=1,2,3,\dots\},
\end{equation}
where $(\mu_{\tilde U,\epsilon,j})_{j\in\N}$ is a non-increasing sequence of real numbers tending to $0$. The numbers $\mu_{\tilde U,\epsilon,j}$, $j\in\N$, are the  eigenvalues of
$S_{\tilde U,\epsilon}^*S_{\tilde U,\epsilon}$, repeated according to their multiplicity. In principle, the sequence $(\mu_{\tilde U,\epsilon,j})_{j\in\N}$ can be ultimately null, but we shall see that this is not the case. Finally, the sequence  $(\mu_{\tilde U,\epsilon,j})_{j\in\N}$ is characterized by the  $\min - \max $ formulae:
\begin{equation}
\mu_{\tilde U,\epsilon,j+1}=\min_{\dim E\leq j}\max_{\substack{U\in\mathcal E^{\perp}\\ \|U\|_{Z_0}=1}}\langle S_{\tilde U,\epsilon}^*S_{\tilde U,\epsilon}U,U\rangle_{Z_0}.
\end{equation}
Let $P_\Sigma$ be the $Z_0$-orthogonal projection onto $\Sigma$. Arguing as in the proof of Theorem XIII.3 in \cite{RS}, we obtain
\begin{equation}\label{ineq-1}
 \sum_{i=1}^d\|\widehat{\partial_u f}(\tilde u)\check\phi_i\|_{L^2}^2 \leq \sum_{i=1}^d \langle S_{\tilde U,\epsilon}^*S_{\tilde U,\epsilon}
 \check\Phi_i,\check\Phi_i\rangle_{Z_0}
 =\mathrm{Tr}(P_\Sigma \circ (S_{\tilde U,\epsilon}^*S_{\tilde U,\epsilon})|_\Sigma)
 \leq \sum_{i=1}^d \mu_{\tilde U,\epsilon,i}.
 \end{equation}
 It follows from (\ref{traceinequality}) and (\ref{ineq-1}) that
\begin{equation}\label{ineq-2}
\mathrm{Tr} (\mathbf B_{\tilde U,\Sigma_d,\delta}) \leq -\frac d \alpha\left(2\nu_\alpha \alpha
-\frac1d \sum_{i=1}^d \mu_{\tilde U,\epsilon,i}\right).
\end{equation}
Now, since $\mu_{\tilde U,\epsilon,i}\to 0$ as $i\to\infty$, the
also the \emph{Cesaro means} $(1/d)\sum_{i=1}^d \mu_{\tilde
U,\epsilon,i}\to 0$ as $d\to\infty$. Therefore there exists
$d=d(\tilde U)$ such that the right-hand side of (\ref{ineq-2}) is
negative. The problem is that $d(\tilde U)$ depends on $\tilde U$,
so we must perform a more careful inspection
 of the asymptotic behavior of the sequence $(\mu_{\tilde U,\epsilon,j})_{j\in\N}$.

Let $(\mu,\Phi)$ be an eigenvalue-eigenvector pair of $S_{\tilde U,\epsilon}^*S_{\tilde U,\epsilon}$, with $\mu\not=0$. This is equivalent to say that
\begin{equation}\label{equal-1}
\langle S_{\tilde U,\epsilon}^*S_{\tilde U,\epsilon}\Phi, U\rangle_{Z_0}=\mu\langle\Phi, U\rangle_{Z_0}\quad\text{for all $U\in Z_0$}.
\end{equation}\label{equal-2}
More explicitly, (\ref{equal-1}) reads
\begin{equation}
\int_\Omega W_{\tilde U,\epsilon}(x)^2\phi u\,dx=\mu\left(\int_\Omega\nabla\phi\cdot\nabla u\,dx+\int_\Omega\beta(x)\phi u\,dx+\int_\Omega \psi v\,dx \right)
\end{equation}
for all $U\in Z_0$, where $\Phi=(\phi,\psi)$ and $U=(u,v)$.
Choosing first $u=0$ and letting $v\in L^2(\Omega)$ be arbitrary, we get that $\psi=0$. It follows that $\phi$ must satisfy
\begin{equation}
\int_\Omega W_{\tilde U,\epsilon}(x)^2\phi u\,dx=\mu\left(\int_\Omega\nabla\phi\cdot\nabla u\,dx+\int_\Omega\beta(x)\phi u\,dx \right)\quad\text{for all $u\in H^1_0$.}
\end{equation}
Thus we have obtained that $(\mu,\Phi)$ is an eigenvalue-eigenvector pair of $S_{\tilde U,\epsilon}^*S_{\tilde U,\epsilon}$ with $\mu\not=0$ if and only if $\psi=0$ and $(\mu,\phi)=(\lambda^{-1},\phi)$,
where $(\lambda,\phi)$ is an eigenvalue-eigenvector pair of the \emph{weighted eigenvalue problem}
\begin{equation}\label{weighted-eigen}
\int_\Omega\nabla\phi\cdot\nabla u\,dx+\int_\Omega\beta(x)\phi u\,dx =\lambda \int_\Omega W_{\tilde U,\epsilon}(x)^2 \phi u\,dx\quad\text{for all $u\in H^1_0(\Omega)$}.
\end{equation}

In order to study (\ref{weighted-eigen}) we proceed  as in
\cite{KaSta}: we denote by $L^2_{W_{\tilde, U\epsilon}}(\Omega)$
the closure of $H^1_0(\Omega)$ with respect to the scalar product
\begin{equation}
\langle u_1,u_2\rangle_{L^2_{W_{\tilde U,\epsilon}}}:= \int_\Omega
W_{\tilde U,\epsilon}(x)^2 u_1 u_2\,dx.
\end{equation}
It turns out that $L^2_{W_{\tilde U,\epsilon}}(\Omega)$ is a
separable Hilbert space, and $H^1_0(\Omega)$ is compactly embedded
into $L^2_{W_{\tilde U,\epsilon}}(\Omega)$. This is a consequence
of the fact that $W_{\tilde U, \epsilon}^2\in L^{r/2}(\Omega)$ with
$r>3$ and  $W_{\tilde U,\epsilon}(x)>0$ a.e. in $\Omega$. The
latter observation makes clear the reason for which we introduced
the correction $\epsilon\rho(x)$. It follows from the general
theory of self-adjoint operators with compact resolvent (see e.g. \cite{EE}) that the
eigenvalues of (\ref{weighted-eigen}), counted according to their
multiplicity, form a sequence $(\lambda_{\tilde U,\epsilon,
j})_{j\in\N}$, with $\lambda_{\tilde U,\epsilon, j}\to +\infty$ as
$j\to\infty$. Now let $\tilde\lambda>0$; we want to find an
estimate for the number $\mathcal N(W_{\tilde
U,\epsilon},\tilde\lambda)$ of eigenvalues of
(\ref{weighted-eigen}) which are strictly smaller than $\tilde\lambda$. To
this end, we exploit a trick due to Li and Yau (see \cite[Cor.
2]{LY}). Namely, we notice that, for $\phi\in H^1_0(\Omega)$,
\begin{multline}\label{rasp}
\frac{\int_\Omega|\nabla\phi|^2\,dx+\int_\Omega\beta(x)\phi^2\,dx
- \int_\Omega \tilde\lambda W_{\tilde U,\epsilon}(x)^2 \phi^2
\,dx}{\int_\Omega \phi^2\,dx}\\ =\frac{\int_\Omega W_{\tilde
U,\epsilon}(x)^2 \phi^2 \,dx}{\int_\Omega
\phi^2\,dx}\left(\frac{\int_\Omega|\nabla\phi|^2\,dx+\int_\Omega\beta(x)\phi^2\,dx}{\int_\Omega
W_{\tilde U,\epsilon}(x)^2 \phi^2 \,dx}-\tilde\lambda\right).
\end{multline}
It follows that, given a finite dimensional subspace $E$ of
$H^1_0(\Omega)$, the expression on the left-hand side in
(\ref{rasp}) is negative on $E$ if and only if the expression on
the right-hand side (\ref{rasp}) is negative on $E$. Now we
observe that the mapping $u\mapsto -\tilde \lambda W_{\tilde
U,\epsilon}(x)^2 u$ is a relatively compact perturbation of
$-\Delta+\beta(x)$. Therefore, by Weyl's Theorem, the essential
spectrum of $-\Delta+\beta(x)-\tilde \lambda W_{\tilde
U,\epsilon}(x)^2$ is contained in $[\lambda_1,+\infty[$. Then,
thanks to the $\min-\max$ characterization of the eigenvalues of
self-adjoint operators (see e.g. \cite{RS}), we deduce that
\begin{equation}\mathcal N(W_{\tilde
U,\epsilon},\tilde\lambda)=n(-\Delta+\beta(x)-\tilde \lambda W_{\tilde U,\epsilon}(x)^2),
\end{equation}
where $n(-\Delta+\beta(x)-\tilde \lambda W_{\tilde U,\epsilon}(x)^2)$ is the number of negative eigenvalues of the operator  $-\Delta+\beta(x)-\tilde \lambda W_{\tilde U,\epsilon}(x)^2$.
The latter can be estimated by mean of Cwickel-Lieb-Rozenblum inequality in its abstract formulation due to Rozenblum and Solomyak (see \cite{RoSo} ). Namely, we have
\begin{equation}\label{CLR}
n(-\Delta+\beta(x)-\tilde \lambda W_{\tilde U,\epsilon}(x)^2)\leq M_r \int_\Omega   ( \tilde \lambda W_{\tilde U,\epsilon}(x)^2)^{r/2}\,dx,
\end{equation}
where $M_r$ is an constant depending only on $r$, $\lambda_1$, $|\beta|_{L^\sigma_{\mathrm u}}$, and on the constant of the embedding $H^2(\Omega)\subset L^\infty(\Omega)$ (see also \cite[Sect. 5]{P2} for details; we stress that the constant $M_r$ can be computed explicitly, though the determination of its optimal value seems out of reach). We have thus obtained that
\begin{equation}
\mathcal N(W_{\tilde U,\epsilon},\tilde\lambda)\leq \tilde\lambda^{r/2}M_r \int_\Omega    W_{\tilde U,\epsilon}(x)^r\,dx.
\end{equation}
Now fix $j\in\N$. For $\tilde\lambda>\lambda_{\tilde U,\epsilon,j}$ we have
\begin{equation}
j\leq N(W_{\tilde U,\epsilon},\tilde\lambda)\leq \tilde\lambda^{r/2}M_r \int_\Omega    W_{\tilde U,\epsilon}(x)^r\,dx.
\end{equation}
By letting $\tilde\lambda$ tend to $\lambda_{\tilde U,\epsilon,j}$ we get
\begin{equation}
j \leq \lambda_{\tilde U,\epsilon,j}^{r/2}M_r \int_\Omega    W_{\tilde U,\epsilon}(x)^r\,dx.
\end{equation}
It follows that
\begin{equation}
 \lambda_{\tilde U,\epsilon,j}\geq M_r^{-2/r}\| W_{\tilde U,\epsilon}\|_{L^r}^{-2}\,j^{2/r},
\end{equation}
whence
\begin{equation}\label{asympt}
 \mu_{\tilde U,\epsilon,j}\leq M_r^{2/r}\| W_{\tilde U,\epsilon}\|_{L^r}^{2}\,j^{-2/r}.
\end{equation}
Putting together (\ref{ineq-2}) and (\ref{asympt}), we get
\begin{equation}
\mathrm{Tr} (\mathbf B_{\tilde U,\Sigma_d,\delta}) \leq -\frac d
\alpha\left(2\nu_\alpha \alpha -\frac1d \sum_{j=1}^d M_r^{2/r}\|
W_{\tilde U,\epsilon}\|_{L^r}^{2}\,j^{-2/r}\right).
\end{equation}
Letting $\epsilon$ tend to $0$ and taking into account (\ref{pot}), we finally get
\begin{equation}\label{ineq-3}
\mathrm{Tr} (\mathbf B_{\tilde U,\Sigma_d,\delta}) \leq -\frac
{M_r^{2/r}\tilde C(\mathcal I)^2 d} \alpha\left(\frac{2\nu_\alpha
\alpha}{M_r^{2/r}\tilde C(\mathcal I)^2} -\frac1d \sum_{j=1}^d
\,j^{-2/r}\right),
\end{equation}
where
\begin{equation}\label{def-Ctilde}
\tilde C(\mathcal I):=\|\partial_u f(\cdot,0)\|_{L^r}+C(1+\sup_{(u,v)\in\mathcal I}\|u\|_{L^\infty})\sup_{(u,v)\in\mathcal I}\|u\|_{L^r}.
\end{equation}
We have thus obtained an estimate for $\mathrm{Tr} (\mathbf
B_{\tilde U,\Sigma_d,\delta}) $ which is uniform with respect to
$\tilde U$ and $\Sigma_d$. Now we are in a position to state and
prove the main result of the paper:
\begin{Theorem}\label{dimension} Assume Hypotheses \ref{hyp1},
\ref{hyp4}, \ref{hyp2} and \ref{hyp3} are satisfied. Let $\mathcal
I\subset Z_0$ be a compact invariant set of the semiflow $\Pi(t)$
generated by (\ref{eq2}). Let $\nu_\alpha$ and $\tilde C(\mathcal I)$ be
defined by (\ref{def-nu}) and (\ref{def-Ctilde}) respectively, and let $M_r$ be the constant of the Cwickel-Lieb-Rozenblum inequality (\ref{CLR}).
Let $d>0$ be such that
\begin{equation}\label{condition1}
\frac1d \sum_{j=1}^d \,j^{-2/r}\leq \frac{\nu_\alpha
\alpha}{M_r^{2/r}\tilde C(\mathcal I)^2}.
\end{equation}
Then the Hausdorff (resp. the fractal) dimension of $\mathcal I$
in $Z_0$ is finite, and is less than or equal to $d$ (resp. $2d$).
\end{Theorem}
\begin{proof}
Let $p_j$, $j\in\N$, be the numbers defined by (\ref{exponents}). If $d$ satisfies condition (\ref{condition1}), then (\ref{ineq-3}) implies that $p_d\leq -\nu_\alpha d$.  Moreover,
for $j=1$, \dots, $d-1$, one has
$$
(p_j)_+\leq \frac
{M_r^{2/r}\tilde C(\mathcal I)^2 } \alpha\sum_{i=1}^{j-1}i^{-2r}\leq \frac
{M_r^{2/r}\tilde C(\mathcal I)^2 } \alpha\sum_{i=1}^{d}i^{-2r}\leq\nu_\alpha d.
$$
It  follows from Proposition \ref{newprop2} and  from the results in \cite[Ch. V, pp 287--291]{Tem} that
$\mathrm{dim}_{H}(\mathcal I)\leq d$ and
$\mathrm{dim}_{F}(\mathcal I)\leq d\max_{1\leq j\leq
d-1}(1+(p_j)_+/|p_d|)\leq 2d$. \end{proof}
\begin{Rem}
We can give an explicit estimate of $d$ just noticing that
\begin{equation}
\frac1d \sum_{i=1}^{d}i^{-2r}\leq\frac1d\int_0^d s^{-2/r}\,ds=\frac{r}{r-2}d^{-2/r}.
\end{equation}
It follows that
\begin{equation}\label{result1}
\mathrm{dim}_{H}(\mathcal I)\leq \left( \frac{r}{r-2}\frac{M_r^{2/r}\tilde C(\mathcal I)^2}{\nu_\alpha\alpha}\right)^{r/2}
\end{equation}
and
\begin{equation}\label{result2}
\mathrm{dim}_{F}(\mathcal I)\leq 2\left( \frac{r}{r-2}\frac{M_r^{2/r}\tilde C(\mathcal I)^2}{\nu_\alpha\alpha}\right)^{r/2}.
\end{equation}
\end{Rem}
Notice that $\nu_\alpha\alpha\to\lambda_1$ as $\alpha\to\infty$. Therefore, if we have a family $\mathcal I_\alpha$ of invariant sets of $\Pi(t)=\Pi_\alpha(t)$ and if we can control $|\mathcal I_\alpha|_{Z^1}$ independently of $\alpha$, we obtain that the dimension of $\mathcal I_\alpha$ remains bounded as $\alpha\to\infty$. This is actually the case when the non-linearity $f$ is dissipative and $\mathcal I_\alpha$ is the compact global attractor of $\Pi_\alpha(t)$, as we shall see in the next section.


\section{Dissipative equations: dimension of the attractor}
In this section we consider the equation
\begin{equation}\label{attr1}
\begin{aligned}
\epsilon u_{tt}+\ u_t+\beta(x)u-\Delta u
&=f(x,u),&&(t,x)\in[0,+\infty[\times\Omega,\\
u&=0,&&(t,x)\in[0,+\infty[\times\partial\Omega
\end{aligned}\end{equation}
where $\epsilon\in]0,1]$. Besides Hypotheses \ref{hyp1},
\ref{hyp4}, \ref{hyp2} and \ref{hyp3}, we assume:

\begin{Hyp}\label{hyp5}There exists a positive number $\mu$ and a function $c(\cdot)\in L^1(\Omega)$ such that:
\begin{enumerate}
\item $f(x,u)u-\mu F(x,u)\leq c(x)$;
\item $F(x,u)\leq c(x)$.
\end{enumerate}
Here, $F(x,u):=\int_0^uf(x,s)\,ds$, $(x,u)\in\Omega\times\R$.
\end{Hyp}
It was proved in \cite{PR2} that, under Hypotheses \ref{hyp1}, \ref{hyp2} and \ref{hyp5}, for every $\epsilon\in]0,1]$ equation (\ref{attr1})  generates a global semiflow in $Z_0$,
possessing a compact global attractor $\mathcal A_\epsilon$. Moreover, there exists a positive constant $R$ such that
\begin{equation*}
\sup_{\epsilon\in]0,1]}\sup\{\|u\|_{H^1_0}^2+\epsilon\|v\|_{L^2}^2\mid(u,v)\in\mathcal A_\epsilon\}\leq R.
\end{equation*}
The constant $R$ depends only  on the constants in Hypotheses \ref{hyp1}, \ref{hyp2} and \ref{hyp5} and on $\|c(\cdot)\|_{L^1}$, and can be explicitly computed (see \cite{PR2}). In particular, $R$ is independent of $\epsilon$. Moreover, it was proved in \cite{P1} that
there exists a positive constant $\tilde R$ such that
\begin{equation*}
\sup_{\epsilon\in]0,1]}\sup\{\|u\|_{H^2\cap H^1_0}^2+\|v\|_{H^1_0}^2\mid(u,v)\in\mathcal A_\epsilon\}\leq \tilde R.
\end{equation*}
Also, the constant $\tilde R$ depends only on the constants in
Hypotheses \ref{hyp1}, \ref{hyp2} and \ref{hyp5} and on
$\|c(\cdot)\|_{L^1}$ and can be explicitly computed (see
\cite{P1}). In particular,  $\tilde R$ is independent of
$\epsilon$. By a time re-scaling ($t=\epsilon^{1/2}s$) we see that
(\ref{attr1}) is equivalent to
\begin{equation}\label{attr2}
\begin{aligned}
\epsilon \check u_{ss}+\alpha\ \check u_s+\beta(x)\check
u-\Delta\check u &=f(x,\check
u),&&(s,x)\in[0,+\infty[\times\Omega,\\ \check
u&=0,&&(s,x)\in[0,+\infty[\times\partial\Omega
\end{aligned}\end{equation}
where $\alpha:=\epsilon^{-1/2}$. Equation (\ref{attr2}) possesses
a compact global attractor $\check{\mathcal A}_\alpha$, such that
\begin{equation}
\check{\mathcal A}_\alpha=\{(\check u,\check v)\in Z_0\mid (\check
u,\alpha\check v)\in \mathcal A_{\alpha^{-2}}\}.
\end{equation}
It follows that $|\check{\mathcal A}_\alpha|_{Z_0}\leq R$ and
$|\check{\mathcal A}_\alpha|_{Z_1}\leq \tilde R$. As a
consequence, the constant $\tilde C(\check{\mathcal A}_\alpha)$ in
(\ref{result1}) and (\ref{result2}) can be explicitly computed,
and in particular it is independent of $\alpha$. We have then
\begin{equation}\label{result3}
\mathrm{dim}_{H}(\mathcal
A_\epsilon)=\mathrm{dim}_{H}(\check{\mathcal
A}_{\epsilon^{-1/2}})\leq \left(
\frac{r}{r-2}\frac{M_r^{2/r}\tilde C(\check{\mathcal
A}_{\epsilon^{-1/2}})^2}{\nu_{\epsilon^{-1/2}}\epsilon^{-1/2}}\right)^{r/2}
\end{equation}
and
\begin{equation}\label{result4}
\mathrm{dim}_{F}(\mathcal
A_\epsilon)=\mathrm{dim}_{F}(\check{\mathcal
A}_{\epsilon^{-1/2}})\leq 2\left(
\frac{r}{r-2}\frac{M_r^{2/r}\tilde C(\check{\mathcal
A}_{\epsilon^{-1/2}})^2}{\nu_{\epsilon^{-1/2}}\epsilon^{-1/2}}\right)^{r/2}.
\end{equation}
Since $\nu_\alpha\alpha\to\lambda_1$ as $\alpha\to\infty$, we obtain that $\mathrm{dim}_{H}(\mathcal A_\epsilon)$ and
$\mathrm{dim}_{F}(\mathcal A_\epsilon)$ remain bounded  as $\epsilon\to 0$,
coherently with the fact that the $\mathcal A_\epsilon$ ``converge", as $\epsilon\to 0$, to the attractor of the parabolic equation
\begin{equation}\label{attr4}
\begin{aligned}
 u_t+\beta(x)u-\Delta u
&=f(x,u),&&(t,x)\in[0,+\infty[\times\Omega,\\
u&=0,&&(t,x)\in[0,+\infty[\times\partial\Omega
\end{aligned}\end{equation}
(see \cite{PR3} and \cite{P1}).


\end{document}